\documentclass[12pt,a4paper, reqno]{amsart}
\usepackage{amsfonts,amsmath,amsthm,amssymb, amscd}
\usepackage[top=2cm, bottom=5cm, left=3cm, right=1cm]{geometry}

\usepackage{xcolor}

\def\GL{\rm GL}
\def\diag{{\rm diag}}

\usepackage{amsthm}
\newtheorem{theorem}{Theorem}

\newtheorem{lemma}{Lemma}

\begin{document}

\title[almost stable range $1$]{Rings of the right (left) almost stable range $1$}

\author{Victor BOVDI}
\address[\texttt{V.Bovdi}]
{United Arab Emirates University, Al Ain, UAE}
\email{vbovdi@gmail.com}

\author{ \framebox{Bohdan Zabavsky}}
\address[\texttt{B.Zabavsky}]
{Department of Mechanics and Mathematics, Ivan Franko National University\\ Lviv, 79000, Ukraine}
\email{zabavskii@gmail.com}


\thanks{The research was supported by UAEU UPAR grants G00002160 and  G00003658.}

\maketitle

\begin{abstract}
{We introduce a concept of  rings of right (left) almost stable range $1$ and  we construct  a theory of  a canonical diagonal reduction of matrices over such rings. A description of  new classes of noncommutative elementary divisor rings is done as well. In particular, for B\'ezout $D$-domain  we introduced the notions of $D$-adequate element and $D$-adequate ring. We proved that every $D$-adequate  B\'ezout domain has almost stable range $1$. For Hermite $D$-ring we proved the necessary and sufficient conditions to be an elementary divisor ring. A ring $R$ is called an $L$-ring if the condition $RaR = R$ for some  $a\in R$ implies that $a$ is a unit  of $R$. We proved  that every $L$-ring of almost stable range $1$ is a ring of right almost stable range $1$.}
\end{abstract}




\bigskip

\centerline{Dedicated to $65^\text{th}$ birthday of Professor Volodymyr Shchedryk}

\section{Introduction}
The stable range  of a ring is one of the most important invariants of the algebraic $K$-theory,
which  was introduced by H.~Bass~\cite{1-bass}. Rings of small stable range  are of particular interest and are being actively studied. Therefore, at present, the theory of rings of stable range  $1$ and $2$ is a rather extensive branch of the theory of rings and modules. However, it should be noted that while studies on stable range $1$ rings are a very popular topic, the same is not true for stable range $2$ rings (see \cite[p.\,2]{3-zabavsky}).

Commutative rings with restrictions on their elements, resembling  the definition of the  stable range,  were considered in \cite{Moore_Steger}. Commutative B\'ezout rings that are of stable range  $1.5$, i.e.~rings in which for any $a\not=0, b,c\in R$ there exists $r$ such that $(a+rb,c )= 1$,  were introduced  in \cite{Shchedryk_2005} and actively studied in \cite{Bovdi_Shchedryk, Shchedryk_2017, 2-schedryk}.

However, it is worth emphasizing that W.~McGovern~\cite{4-mcGovern} is credited with the first study of commutative rings with the properties that non-trivial homomorphic images are a ring of stable range $1$, i.e.~rings of almost stable range $1$. This ground breaking study has given impetus to further research in this area.  The historical overview and current state of the topic are covered in books  \cite{Mon_Shchedryk, 12-zabavsky} and the articles \cite{Bovdi_Zabavsky, Bovdi_Shchedryk,  Gatalevich, Calugareanu,  Moore_Steger, Dopico}.


In the present  article, we introduce a notion of non-commutative rings with almost stable range $1$. We show that every ring of stable range  $1$ is a ring of right (left) almost stable range $1$. In addition, the concept of an element of right (left) almost stable range $1$ is introduced. Based on the above, we introduce  a notion of rings of right (left) almost stable range $1$ and study their properties and connections with other ring properties.

Let $R$ be an associative ring (not necessary commutative) with $1\not=0$, let $U(R)$ be the group of units of $R$  and let $R^{n\times m}$ be the vector space  of ${n\times m}$ matrices over $R$ with $n,m\geq 1$. Let $\GL_n(R)$ be the group of units of the matrix ring $R^{n\times n}$.

The matrix $D:=\diag(d_{1},\ldots,d_{s})\in R^{n\times m}$ means a (possibly rectangular) matrix  having $d_{1},\ldots, d_{s}$ (in which $s:=\min(n,m)$) on the main diagonal and zeros elsewhere.  By the main diagonal we mean the one beginning at the upper left corner.

Two matrices $A$ and $B$ over a ring $R$ are called {\it equivalent}  if there exist invertible matrices $P$ and $Q$ over $R$ of  suitable sizes such  that $A=PBQ$ and it  is denoted by $A\sim B$.

According to I.~Kaplansky (see \cite[p.\,465]{7-kaplansky}),    a ring $R$ is called an {\it elementary divisor ring} if for any matrix $A\in R^{n\times m}$ we have
\begin{equation}\label{Eq:1}
A\sim D:=\diag(d_1,\ldots,d_k,0,\ldots,0),
\end{equation}
in which  $d_{i}$ is a {\it total divisor} of $d_{i+1}$, i.e.~$Rd_{i+1}R\subseteq d_iR\cap Rd_i$ for each $i=1,\ldots, k-1$.
In this case, we say that  the matrix $A$ has a  {\it canonical diagonal reduction} over $R$.

The class of elementary divisor rings is contained in the class of B\'ezout rings (for example, see \cite{7-kaplansky, Mon_Shchedryk, 12-zabavsky}),
i.e.~rings with nonzero unit in which every finitely generated one-sided ideal is a principal one-sided ideal.
Note that elementary divisor rings are Hermite rings, i.e.~rings in which each  $1\times 2$ and $2\times 1$ matrix   has a  diagonal reduction,
i.e.~$(a,b)P=(c,0)$ and \linebreak
$Q(a, b)^T=(d,0)^T$, where    $a, b, c, d\in R$ and   $P, Q\in \GL_2(R)$  (see \cite{7-kaplansky, 12-zabavsky}).  Every  right Hermite ring is a  right B\'ezout ring.

\section{Rings of almost stable range $1$}

Set
$R^2:=R\times R=\big\{(a,b) : a,b\in R\big\}$  and $R^3:=\big\{(a,b,c) : a,b,c\in R\big\}$.

A ring  $R$ has  {\it stable range $1$} if for each pair $(a,b)\in R^2$ the equality $aR+bR=R$ implies $(a+b\lambda)R=R$ for some $\lambda\in R$.
Similarly, a ring  $R$ has  {\it stable range $2$} if for each triple \linebreak
$(a,b,c)\in R^3$ the equality $aR+bR+cR=R$ implies that there  exist $\lambda,\mu\in R$ such that
\[
(a+c\lambda)R+(b+c\mu)R=R.
\]

An element  $a\in R\setminus\{0\}$ has  {\it   right (left) almost stable range $1$}, if for each pair $(b,c)\in R^2$  there exists $\lambda\in R$ ($\mu\in R$) such that  equality  $aR+bR+cR=R$  implies   $aR+(b+c\lambda)R=R$ ($Ra+Rb+Rc=R$ implies  $Ra+R(b+\mu c)=R$, respectively). If each nonzero element of $R$ is an element of right (left) almost stable range $1$, then $R$ is a ring of  {\it right (left) almost stable range~$1$}.

This notion generalizes the  notion of rings of almost stable range $1$, which was introduced by W.~McGovern for commutative rings.
 A right (left) B\'ezout ring  is a ring in which every right (left) finitely generated ideal is a principal right (left) ideal.  A B\'ezout ring is a ring that is both right and left B\'ezout.

Let $(a,b)\in R^2$ be such that  $aR+bR=R$. The pair $(a,b)\in R^2$ is called   {\it right diadem}, \linebreak
if there  exists  $\lambda\in R$ such that for each  triple $(a+b\lambda, c,d)\in R^3$ with the property  \linebreak
$(a+b\lambda)R+cR+dR=R$ there exists  $\mu\in R$ such that
$(a+b\lambda)R+(c+d\mu)R=R$.


The  concept of a {\it  left diadem}  is  introduced by analogy.
A {\it diadem} is a pair that is both right and  left diadem at the same time. We say that $R$ is a ring of {\it right dyadic range $1$} if for each pair $(a,b)\in R^2$ the equality $aR+bR=R$ implies that the pair $(a,b)$ is a right diadem. Similarly we define  rings of left dyadic range $1$. A ring  {\it of dyadic range $1$} is a ring that is both of right dyadic range $1$ and of  left dyadic range $1$ at the same time (see \cite[Definition~2.1 and Definition~2.3]{3-zabavsky}).
Note that if $R$  has  right almost stable range $1$,  then each pair $(a,b)\in R\setminus\{0\}\times R$ is  a right diadem,   because
$aR+(c+b\mu)R=R$, $\lambda=0$, $\mu\in R$.

%
%

\begin{lemma}\label{Lemma:1}
Let $R$ be a right B\'ezout ring.
If\/ $R$  has  stable range $2$,  then for each  pair $(a,b)\in R^2$ there exists  a triple $(d,a_1,b_1)\in R^3$ such that
\[
a=da_1,\qquad  b=db_1 \qquad \text{and}\qquad  a_1R+b_1R=R.
\]
\end{lemma}
\begin{proof}
Since $R$ is a right B\'ezout ring, $aR+bR=dR$ for some $d\in R$. Clearly,   $a=da_0$, \linebreak
$b=db_0$ and $au+bv=d$ for some $d,a_0,b_0,u,v\in R$.
Thus   $da_0u+db_0v=d$ and   $dc=0$, where  \linebreak
$c:=1-a_0u-b_0v$.  Consequently $a_0R+b_0R+cR=R$. Since $R$ has  stable range $2$, we have
$(a_0+c\lambda)R+(b_0+c\mu)R=R$
for some $\lambda,\mu\in R$.
Finally, setting   $a_1:=a_0+c\lambda$ and $b_1:=b_0+c\mu$,  we obtain $a_1R+b_1R=R$,  $a=da_1$ and  $b=db_1$.
\end{proof}

\begin{theorem}\label{T:1}
The following conditions hold.
\begin{itemize}
\item[(i)] Each right B\'ezout ring of stable range $1$ is a ring of right almost stable range $1$.
\item[(ii)] Each ring of right (left) almost  stable range $1$ is a ring of right (left)  dyadic  range $1$.
\item[(iii)] Each right (left) B\'ezout ring of a right (left) almost stable range $1$ is a ring of stable  range~$2$.
\end{itemize}
\end{theorem}

\begin{proof}
Let $aR+bR+cR=R$ and $bR+cR=dR$. Each ring of stable range $1$ is a ring of stable range $2$ (see \cite[p.\,14]{1-bass}), so  $b=db_1$, $c=dc_1$ and $c_1R+b_1R=R$ for some $d,c_1,b_1\in R$ by Lemma~\ref{Lemma:1}. Since $R$ has  stable range $1$, from $c_1R+b_1R=R$ it follows that  $b_1\lambda+c_1=u\in U(R)$ for some $\lambda\in R$. That yields $db_1\lambda+dc_1=du$, i.e.~$b\lambda+c=du$. Since $aR+bR+cR=R$ and $bR+cR=dR$, we have
\[
aR+duR=aR+bR+cR=aR+(b\lambda+c)R=R,
\]
i.e.~$R$ is a ring of right almost stable range $1$.

Each B\'ezout ring of right dyadic range $1$ is a ring of stable range $2$ (see  \cite[Theorem~2.1]{3-zabavsky}). Two other statements are obvious.
\end{proof}

\begin{theorem}\label{T:2}
A ring $R$ has  right almost stable range $1$ if and only
if\/ $R/J(R)$ has right almost stable range $1$, where $J(R)$ is the Jacobson radical of\/ $R$.
\end{theorem}
\begin{proof} Let $\overline{x}:=x+J(R)\in \overline{R}:=R/J(R)$, where $x\in R$. Let $(a,b,c)\in R^3$ be such that
\[
\overline{a\vphantom{b}}\,\overline{R\vphantom{b}}+\overline{b}\,\overline{R\vphantom{b}}
+\overline{c\vphantom{b}}\,\overline{R\vphantom{b}}=\overline{R\vphantom{b}}
\quad\text{and}\quad a\notin J(R).
\]
That  yields   $aR+bR+cR=R$ and $a\ne0$. Since $R$ has  right almost stable range $1$, we get
\[
aR+(b\lambda +c)R=R, \quad \lambda\in R.
\]
Thus  $\overline{a\vphantom{b}}\,\overline{R\vphantom{b}}+(\overline{b}\,\overline{\lambda}
+\overline{c\vphantom{b}})\overline{R\vphantom{b}}=\overline{R\vphantom{b}}$,
so   $\overline{R\vphantom{b}}$ has  right almost stable range $1$. The rest of the proof is similar.
\end{proof}

\begin{theorem}\label{T:3}
Let $R$ be a right B\'ezout ring.  If for each  $(a,b)\in R^2$ with the property $aR+bR=R$ there exists
$\lambda\in R$ such that $a+b\lambda$ is an element of right almost stable range $1$, then $R$ is a ring of stable range $2$.
\end{theorem}

\begin{proof}
Let $(a,b,c)\in R^3$ such that $aR+bR+cR=R$. Since $R$ is a right B\'ezout ring, we have $bR+cR=dR$ for some  $d\in R$, so $aR+dR=R$. By our assumption there exists $\lambda\in R$ such that the element  $\nu:=a+d\lambda$  has  right almost stable range $1$.   From  $bR+cR=dR$ we obtain  that $a+bx+cy=\nu$ for some $x,y\in R$, so $\nu R+bR+cR=R$, because   $aR+bR+cR=R$.

By the  definition of an element of   right almost stable range $1$ we have $\nu R+(b+c\mu)R=R$ for some $\mu\in R$. Let $\nu s+(b+c\mu)t=1$ for some $s, t\in R$. Since $Rs+Rt=R$, we have
$Rs+R(xs+t)R=R$.

 Let $us+v(xs+t)=ys+\mu t$ for some $u,v\in R$. Thus
$(a+cu)s+(b+cv)(xs+t)=1$,
i.e.~$(a+cu)R+(b+cv)R=R$, so  $R$ is a ring of stable range $2$.
\end{proof}

\section{Case of commutative rings}

Let $R$ be a commutative  ring. Then in \eqref{Eq:1} each   $d_{i}$ is a divisor of $d_{i+1}$  for $i=1,\ldots, k-1$.  According to
W.~McGovern  \cite[p.\,393]{4-mcGovern}, we say that a commutative ring $R$ has  almost stable range $1$, if its   every proper homomorphic image has stable range $1$. We have   the following assertion.

\begin{theorem}\label{T:4}
For  a commutative ring $R$ the following statements are equivalent:
\begin{itemize}
      \item[(i)] $R$ is a ring of almost stable range $1$,
      \item[(ii)] for every $(a,b,c)\in R^3$ with the properties  $a\ne0$ and $aR+bR+cR=R$,  there exists  $\lambda\in R$ such that $aR+(b+c\lambda)R=R$.
    \end{itemize}
\end{theorem}

\begin{proof}
The proof of the  part $\textsl{(i)} \Rightarrow \textsl{(ii)}$  can be found in \cite[Theorem~3.6]{4-mcGovern} putting  $a\ne 0$.

Let $(a,b,c)\in R^3$ such that  $aR+bR+cR=R$ and $a\ne0$. Assume  there exists  $\lambda\in R$ such that $aR+(b+c\lambda)R=R$.

Set  $\overline{R\vphantom{b}}:=R/aR$ and $\overline{x}:=x+aR$, where $x\in R$. If $(b,c)\in R^2$ such that
$\overline{b}\,\overline{R\vphantom{b}}+\overline{c\vphantom{b}}\,\overline{R\vphantom{b}}=\overline{R\vphantom{b}}$,
then $aR+bR+cR=R$. According to the assumption
we have $aR+(b+c\lambda)R=R$ for some $\lambda\in R$.
Evidently $(\overline{b}+\overline{c\vphantom{b}}\,\overline{\lambda})\overline{R\vphantom{b}}=\overline{R\vphantom{b}}$,
where $\overline{b}=b+aR$, i.e.~$\overline{R\vphantom{b}}$ is a ring of stable range $1$. If $I$ is a nonzero ideal of $R$, then for every  $0\not=a\in I$ we have the isomorphism $R/I\big /aR/I\simeq  R/aR$. If $R/aR$ is a ring of stable range $1$, then  $R/I$ has  stable range $1$ as well.
\end{proof}

A ring  is called {\it clean} if its every    nonzero element   is  a sum of a unit and an idempotent \cite{Nicholson_77}.
A ring  in  which every  proper homomorphic image is clean   is called {\it neat} (see \cite[p.\,248]{6-mcGovern}).
An element $a\in R$ is called  {\it neat} if $R/aR$ is a clean ring.

An $R$-module $M$  has {\it finite exchange property} if for any $R$-module $N$ with the  decomposition
$N=M'\oplus P=\bigoplus_{i\in I}Q_i$ in which  $M'\cong M$ and $I$ is finite, there exist submodules $Q_i'\subseteq Q_i$ for each $i\in I$, such that $N=M'\oplus(\oplus Q_i')$. A ring $R$ is called {\it exchange} if its regular module $R_R$  is a  finite exchange $R$-module.

A commutative ring $R$ has {\it  neat range $1$} if for each pair  $(a,b)\in R^2$ with the property   \linebreak
$aR+bR=R$ there exists a neat element $c\in R$ such that $a+bt=c$ for some $t\in R$. For the properties of the above rings and their relationship with other types of rings, see \cite{Vas, Lam_2, Lam_3, 6-mcGovern, 10-zabgat, 9-zabavsky}.

\begin{theorem}\label{T:5}
Let $R$ be a commutative B\'ezout ring. If for each  $(a,b,c)\in R^3$  with the properties   $aR+bR+cR=R$ and $a\in R\backslash\{0\}$ there exists a decomposition $a=rs$  for some $r,s\in R$  such that
\begin{equation}\label{Edvvr:1}
rR+bR=R,\qquad  sR+cR=R\qquad  \text{and}\qquad  rR+sR=R,
\end{equation}
then $R/aR$ is a clean ring.
\end{theorem}

\begin{proof}
Set $\overline{R}:=R/aR$ and $\overline{x}:=x+aR$, where $x\in R$. Since  $aR+bR+cR=R$ and \eqref{Edvvr:1} holds, we get
\[
\overline{b}\,\overline{R\vphantom{b}}+\overline{c\vphantom{b}}\,\overline{R\vphantom{b}}=\overline{R\vphantom{b}},
\qquad
\overline{r\vphantom{b}}\,\overline{R\vphantom{b}}+\overline{s\vphantom{b}}\,\overline{R\vphantom{b}}=\overline{R\vphantom{b}},
\qquad
\overline{b}\,\overline{R\vphantom{b}}+\overline{r\vphantom{b}}\,\overline{R\vphantom{b}}=\overline{R\vphantom{b}}
\qquad
\text{and}
\qquad
\overline{c\vphantom{b}}\,\overline{R\vphantom{b}}+\overline{s\vphantom{b}}\,\overline{R\vphantom{b}}=\overline{R\vphantom{b}}.
\]
It follows that  $\overline{r}\,\overline{u}+\overline{s}\,\overline{v}=\overline{1}$ and $\overline{r}\,\overline{s}=\overline{0}$ for some $u,v\in R$. This yields  $\overline{r}^2\overline{u}=\overline{r}$ and  $\overline{s}^2\overline{v}=\overline{s}$. Obviously,
\[
(\overline{s}\,\overline{v})^2=\overline{s}\,\overline{v}=\overline{e}\quad\text{and} \quad  (\overline{r}\,\overline{u})^2=\overline{r}\,\overline{u}=\overline{1}-\overline{e}.
\]

\smallskip

\noindent
Taking into account that  $\overline{r\vphantom{b}}\,\overline{R\vphantom{b}}+\overline{b}\,\overline{R\vphantom{b}}=\overline{R\vphantom{b}}$,
we get
\[
\overline{r\vphantom{b}}\,\overline{x\vphantom{b}}+\overline{b}\,\overline{y\vphantom{b}}=\overline{1\vphantom{b}}
\]
for some $\overline{x},\overline{y}\in \overline{R}$.
Thus
$\overline{s\vphantom{b}}\,\overline{v\vphantom{b}}\,\overline{r\vphantom{b}}\,\overline{x\vphantom{b}}
+\overline{s\vphantom{b}}\,\overline{v\vphantom{b}}\,\overline{b}\,\overline{y\vphantom{b}}
=\overline{s\vphantom{b}}\,\overline{v\vphantom{b}}$ and
$\overline{e\vphantom{b}}\,\overline{b}\,\overline{y\vphantom{b}}=\overline{e\vphantom{b}}$,
i.e.~$\overline{e\vphantom{b}}\in \overline{b}\,\overline{R\vphantom{b}}$.
Using the fact that  $\overline{s\vphantom{R}}\,\overline{R}+\overline{c\vphantom{R}}\,\overline{R}=\overline{R}$,
we obtain   $\overline{1}-\overline{e\vphantom{R}}\in \overline{c\vphantom{R}}\,\overline{R}$, so  $\overline{R}$ is an exchange ring  by \cite[Proposition 1.1 and Theorem 2.1]{7-kaplansky}.  Finally,  $\overline{R}$ is a clean ring by \cite[p.\,244]{6-mcGovern}.
\end{proof}

In the case when $R$ is a domain,  the following result  is known.

\begin{theorem}[{\!\!\cite[Theorem 31]{9-zabavsky}}]\label{T:6}
Let $R$ be a commutative B\'ezout domain and let $a\in R\backslash\{0\}$. If $R/aR$ is a clean ring,  then for  each $(b,c)\in R^2$ with the property $aR+bR+cR=R$ there exists  a decomposition $a=rs$  for some $r,s\in R$ such that
\[
rR+bR=R,\qquad   sR+cR=R\quad \text{and}\quad  rR+sR=R.
\]
\end{theorem}

Now, using  Theorems~\ref{T:5} and  \ref{T:6} we have the following.

\begin{theorem}\label{T:7}
Let $R$ be a commutative B\'ezout domain. The following statements hold:
\begin{itemize}
\item[(i)] $R$ has neat range\/ $1$ if and only if for each $(a,b)\in R^2$ with the property  $aR+bR=R$ there exists $\lambda\in R$ such that $R/(a+b\lambda)R$ is clean;

\item[(ii)] $R$ is an elementary divisor ring if and only if\/ $R$ has neat range $1$ (see \cite[Theorem 33]{9-zabavsky});

\item[(iii)] if the stable range of\/ $R/aR$ is not\/ $1$ for each $a\in R\setminus\{0\}$,  then $R$ is not an elementary divisor ring
(see \cite[Theorem 5]{10-zabgat}).
\end{itemize}
\end{theorem}

Each  commutative clean ring $R$ has stable range $1$ (see  \cite[p.\,244]{6-mcGovern} and \cite[Proposition 1.8]{Nicholson_77}).
Furthermore, a neat ring is a ring of almost stable range $1$ and a neat element is an element of almost stable range $1$, so  if $R$ is an elementary divisor domain that is not a ring of stable range~$1$, then at $R\setminus U(R)$ there exists a neat element that has  almost stable range $1$
(see \cite[Theorem~7]{11-zabavsky}).

\section{Rings with Dubrovin property}

The {\it coboundary}  of a one-sided ideal $I$ of a ring $R$ is a  two-sided ideal which equals to the intersection of all two-sided ideals which contain $I$.     Note that this definition is  left-right symmetric.

A ring $R$ in  which for every   $a\in R\setminus\{0\}$ there exists  $b\in R$ such that
\begin{equation}\label{OOOjggT}
RaR=bR=Rb,
\end{equation}
(in other words, the  coboundary of $R$ is a   principal ideal) is called  a ring with {\it Dubrovin property} ($D$-property). Simple elementary divisor rings \cite[Section 4.2]{12-zabavsky},  quasi-duo elementary divisor rings \cite[Theorem 1]{15-Bowtell_Cohn} and  semi-local   semi-primitive elementary divisor rings \cite[Theorem 1]{14-dubrovin},
\cite{dubrovin_plus}  are examples of elementary divisor rings with $D$-property.

In \cite{Bovdi_Zabavsky_2}, the investigation of elementary divisor rings of stable range $1$ with Dubrovin and
Dubrovin-Komarnytsky property was done.
In the same article, a theory of a canonical diagonal reduction of matrices over such rings was  constructed as well,  and new families of non-commutative rings of elementary divisor rings  were  presented.

A two-sided analogue of one-sided rings of almost stable range $1$ is the following.

A  ring $R$ has {\it almost stable range $1$} if for each triple $(a,b,c)\in R^3$ with the properties
\[
RaR+RbR+RcR=R \quad \text{and}\quad c\ne 0,
\]
there exists $\lambda\in R$ such that $R(\lambda a+b)R+RcR=R$.

\begin{theorem}\label{T:8}
Let $R$ be a B\'ezout ring. The following statements are equivalent:
\begin{itemize}
\item[(i)]$R$ is a ring of almost stable range $1$;
\item[(ii)] for each $(a,b,c)\in R^3$ with the properties   $RaR+RbR+RcR=R$ and  $c\ne0$, there exist $\lambda,d\in R$ such that
$(\lambda a + b)R + cR = dR$ and $RdR = R$.
\end{itemize}
\end{theorem}

\begin{proof}
Let $RaR+RbR+RcR=R$ and $(\lambda a +b)R+cR=dR$ in which  $RdR=R$ for some $\lambda\in R$.
Since $cR\subset RcR$ and $(\lambda a +b)R\subset R(a\lambda +mb)R$, we have
\[
dR=(\lambda a+b)R+cR\subseteq R(\lambda a+b)R+RcR.
\]
Hence, $RdR=R$ and $d\in R(\lambda a +b)R+RcR$, so  we have $R(\lambda a +b)R+RcR=R$.

Let the condition $RaR+RbR+RcR=R$ imply  that $R(\lambda a +b)R+RcR=R$ with  $c\ne0$  and  some $\lambda\in R$. It is easy to see that $(\lambda a +b)R+cR=dR$ for some $d\in R$,  because $R$ is a B\'ezout ring. Using the fact that  $(\lambda a +b)R\subseteq dR$, we obtain that $cR\subseteq dR$, so  $RcR\subseteq RdR$ and $R(\lambda a +b)R\subseteq RdR$. Consequently  from  $R(\lambda a +b)R+RcR=R$ we obtain that $RdR=R$, which
proves the equivalence of \textsl{(i)} and \textsl{(ii)}.
\end{proof}

A ring $R$ is called an {\it $L$-ring}      if the condition $RaR=R$ for some  $a\in R$ implies that  $a$ is a unit of $R$. This definition arise from   conditions of \cite[Proposition 7.3(2)]{19-lamdugas}.

\begin{theorem}\label{T:9}
Every  $L$-ring of almost stable range $1$ is a ring of  right almost stable range $1$.
\end{theorem}

\begin{proof}
Let $aR+bR+cR=R$ and $c\ne 0$. Obviously $RaR+RbR+RcR=R$ and  $R$ is a ring of  almost stable range $1$ by Theorem ~\ref{T:8}.
Thus, $(\lambda a +b)R+cR=dR$ in which  $RdR=R$. Since $R$ is an  $L$-ring,  $(\lambda a +b)R+cR=R$, i.e.~$R$ is a ring of a right almost stable range $1$.
\end{proof}

Note that every  elementary divisor $L$-ring  is a ring with $D$-property  \cite[Theorem~4.7.1]{12-zabavsky}.
According to Theorem~\ref{T:1} and Theorem~\ref{T:9}  we have the following.

\begin{theorem}\label{T:10}
Every  B\'ezout $L$-ring  of almost stable range $1$ has  stable range $2$.
Moreover, each B\'ezout ring $R$ of stable range $1$ has  almost  stable range $1$.
\end{theorem}

\begin{proof}
The first part follows from  Theorems~\ref{T:1} and \ref{T:9}.

Let $RaR+RbR+RcR=R$ and $Ra+Rb=Rd$. Since a  ring of stable range $1$ has  stable range $2$ (see \cite[p.\,14]{1-bass}),
we have   $a=a_0d$,  $b=b_0d$ and $Ra_0+Rb_0=R$ by Lemma~\ref{Lemma:1}. The ring  $R$ has  stable range $1$ and $Ra_0+Rb_0=R$, so  $a_0+\lambda b_0=u\in U(R)$. Then $a+\lambda b=ud$, i.e. $R(a+\lambda b)=Rud=Rd$ and $a=a_0d$,  $b=b_0d$. Let $dR+cR=zR$. Similarly we have $dy+c=z$. Then we have $(a+\lambda b)R+cR=zR$. Since $a=a_0d+a_0zd_0$,  we have $b=b_0d=b_0zd_0$ and  $c=c_0z$ for some $d_0,c_0\in R$. It follows that  $RaR\subseteq RzR$,
$RbR\subseteq RzR$, and $RcR\subseteq RzR$. Since $RaR+RbR+RcR=R$, we have   $R=RzR$ and  $R$ has  almost stable range $1$.
\end{proof}

\section{Rings with $D$-adequacy property}

The notion of the adequacy of commutative domains was  introduced by  O.~Helmer \cite{Helmer}.
A.I.~Gatalevych~\cite[Definition 1, p.\,116]{Gatalevich} was the first who extended this notation  to non-\-com\-mu\-ta\-tive  rings.
Now it is known that  the generalized right adequate (in the sense of Gatalevych)  duo B\'ezout domain is an elementary divisor  domain  (see  \cite[Theorem 2, p.\,117]{Gatalevich}).

In \cite{Shchedryk_Gatalevich}, it was explored the problem when  a ring of matrices over either adequate rings or elementary divisor rings inherits the property of adequacy.

Let  $A$ and $E$ be   adequate and elementary divisor domains, respectively. Let  $A^{2 \times 2}$ and $E^{2 \times 2}$  be    rings of ${2 \times 2}$ matrices over rings  $A$ and $E$, respectively. In \cite{Shchedryk_Gatalevich}, it was  proved that the set of full nonsingular matrices from $A^{2 \times 2}$ is an  adequate  set in $A^{2 \times 2}$ and the set of full singular matrices from $E^{2 \times 2}$ is an  adequate set  in the set of full matrices in $E^{2 \times 2}$.

In \cite{Bovdi_Shchedryk_adecvate}, another definition of adequate rings was proposed,
which differs from the one  proposed by A.I.~Gatalevych~\cite[Definition 1, p.\,116]{Gatalevich}.

We  propose a  new version of the definition of adequacy, which is very close to the definitions of adequacy that were given in \cite{Gatalevich, Shchedryk_Gatalevich,Bovdi_Shchedryk_adecvate}.

The introduction of new definitions is related to Dubrovin's condition for non-commutative rings. For a class of duo rings our definition is equivalent to the definition of A.I.~Gatalevich \cite[Definition 1, p.\,116]{Gatalevich}.

Let $R$ be a B\'ezout $D$-domain and  let $a\in R\setminus\{0\}$ be such that
\[
RaR=a^*R=Ra^*\ne R.
\]

The element $a$ is called  {\it $D$-adequate} if for each $b\in R$  the following conditions hold:
\begin{itemize}
\item[(i)] $a=rs$ for some $r,s\in R$ such that  $RrR=r^*R=Rr^*$, $RsR=s^*R=Rs^*$, and  \linebreak $RbR=b^*R=Rb^*$;
\item[(ii)]$r^*R+b^*R=R$;
\item[(iii)]for each non-trivial divisor $s^{\prime *}$ of $s^*$ we have $s'^*R+b^*R\ne R$ in which  $Rs'R=s'^*R=Rs'^*$.
\end{itemize}

A ring $R$ is called {\it $D$-adequate} if each  element $0\not=a\in R$ with the property  $RaR\ne R$  is $D$-adequate.

\begin{theorem}\label{T:11}
Every    $D$-adequate B\'ezout $D$-domain  has  almost stable range $1$.
\end{theorem}

\begin{proof}
Let  $(a,b,c)\in R^3$ such that  $RaR+RbR+RcR=R$ and  $c\ne 0$. If $RcR=R$, then the statement of our  theorem is obvious. Let $RcR\ne R$.
It follows that
\begin{itemize}
\item[(i)] $c=rs$;
\item[(ii)] $r^*R+a^*R=R$;
\item[(iii)] for each non-trivial divisor $s'^*$ of $s^*$ we have $s'^*R+b^*R=R$.
\end{itemize}

Let $(a+rb)R+cR=dR$. Using the fact that  $RdR=d^*R=Rd^*$, we obtain \linebreak
$dR\subset RdR=d^*R=Rd^*$. Hence, we obtain $(a+rb)R\subset dR\subset d^*R$. If
\[
r^*R+d^*R=hR\subset h^*R=Rh^*,
\]
then  $aR\subset h^*R\ne R$. Clearly  $h^*R$ is a two-sided ideal, so  $a^*R\subset h^*R$, which contradicts to (ii).
Hence $r^*R+d^*R=R$, so  $r^*u+d^*v=1$ and $r^*us^*+d^*vs^*=s^*$ for some $u,v\in R$. Since $s^*$ is a duo element, then $r^*s^*u'+d^*s^*v'=s^*$ for some $u',v'\in R$. It follows that
\[
sR\subset s^*R\subset d^*R\qquad \text{and}\qquad  d^*R+aR=h^*R\ne R.
\]
It follows that $aR\subset h^*R$, $bR\subset h^*R$, $cR\subset h^*R$ and $h^*R\ne R$. Since $h^*R=Rh^*$, we get  $RaR\subset h^*R$, $RbR\subset h^*R$, and  $RcR\subset h^*R$, which contradicts the condition \[
RaR+RbR+RcR=R.
\]
Thus, we proved that
$(a+rb)R+cR=dR$ and $RdR=R$, i.e.~$R$ has almost stable range $1$.
\end{proof}

\section{ Hermite $D$-rings}

We denote by $A_*$ a two-sided ideal in a ring $R$ generated by all elements of  the matrix $A=(a_{ij})\in R^{m\times n}$.
Evidently $A_*=\sum\limits_{i=1}^m\sum\limits_{j=1}^n Ra_{ij}R$.

We need the   following  well-known results.
\begin{lemma}\label{Lemma:2}
If $A\sim B$, then $A_*=B_*$.
\end{lemma}
\begin{proof}
Let $A=(a_{ij}), B=(b_{ij})\in R^{m\times n}$ be such that $A\sim B$. It follows that  $a_{ij}\in \sum_k\sum_sRb_{ks}R$ and $b_{ij}\in \sum_k\sum_sRa_{ks}R$ for each $i$, $j$, $k$, $s$. Consequently
\[
    \sum_i\sum_jRa_{ij}R= \sum_k\sum_sRb_{ks}R,
\]
i.e.~$A_*=B_*$.
\end{proof}



\begin{lemma}\label{Lemma:3}
Each  Hermite ring  $R$ with   $D$-property  is an elementary divisor ring if and only if every  $A=(a_{ij})\in R^{n\times n}$ with the property   $\sum\limits_{i}\sum\limits_{j} Ra_{ij}R=R$ has a canonical diagonal reduction.
\end{lemma}

\begin{proof} Since the proof of the ``if'' part  is obvious, we start with the proof of the ``only if'' part.

Let $Ra_{ij}R= a_{ij}^*R= Ra_{ij}^*R$ for  each $i$, $j$ (see \eqref{OOOjggT}). For some duo-element $\alpha\in R$ we have
\[
\sum\limits_{i}\sum\limits_{j} Ra_{ij}R=\sum\limits_{i}\sum\limits_{j} a_{ij}^*R=\sum\limits_{i}\sum\limits_{j} Ra_{ij}^*=\alpha R=R\alpha.
\]

That yields  $a_{ij}=\alpha a_{ij}^0$ and $ A=\mathrm{diag}(\alpha,\dots,\alpha)A_0$, where $A_0= (a_{ij}^0)$.
Since $R$ is Hermite, it  is a B\'ezout ring of stable range $2$ \cite[p.\,30, Collorarly $2$.1.3, Theorem $1$.2.40]{12-zabavsky}, so  $\sum_i\sum_jRa_{ij}'R=R$ by Lemma~\ref{Lemma:1}.
\end{proof}

\begin{theorem}\label{T:12}
Let $R$ be a Hermite ring of almost stable range $1$.  If $a,b,c\in R$ with the property  $RaR+RbR+RcR=R$, then there exist $P, Q\in \GL_{2}(R)$ such that
\[
P\left[
\begin{array}{cc}
a & 0 \\
b & c\\
\end{array}\right]Q=\left[
\begin{array}{cc}
z & 0 \\
* & *\\
\end{array}
\right]\qquad \text{and}\qquad   RzR=R.
\]
\end{theorem}

\begin{proof}
Since every  Hermite ring is B\'ezout  \cite[Corolary 2.13]{12-zabavsky},
then $R$  has almost stable range $1$.  For each $(a,b,c)\in R^3$ with the property  $RaR+RbR+RcR=R$,  there exists $\lambda \in R$ such  that $(\lambda a+b)R+cR=dR$ by Theorem ~\ref{T:8}. This yields that
\[
\left[\begin{array}{cc}
\lambda&1\\
1&0\\
\end{array}
\right]
\left[\begin{array}{cc}
a&0\\
b&c\\
\end{array}\right]=\left[\begin{array}{cc}\lambda a+b& c\\a&0\end{array}\right].
\]
Set $d:=(\lambda a+b)x+cy$,  $a_0d:=\lambda a+b$ and  $c:=c_0d$ for some $a_0,c_0,x,y\in R$. Each Hermite ring is a B\'ezout ring of stable range $2$ \cite[Theorem~1.2.40]{12-zabavsky}, so  according to Lemma~\ref{Lemma:1} we have that $\lambda a+b=a_1d$ and  $c=c_1d$, where $a_1R+c_1R=R$,
i.e.~$a_1u+c_1v=1$ for some $u,v\in R$. Since $Ru+Rv=R$ and $R$ is a Hermite ring, then the  column
$\left[\begin{array}{c}
u \\
v\\
\end{array}\right]$
is completable   to a  matrix $
Q=\left[\begin{array}{cc}

        u & * \\
        v & *
\end{array}\right]\in \GL_2(R)$ (see  \cite[Corolary 2.1.7]{12-zabavsky}). Evidently
\[
\left[\begin{array}{cc}\lambda& 1\\1&0\end{array}\right]\left[\begin{array}{cc}a&0\\b&c\end{array}\right]Q=\left[\begin{array}{cc}d& *\\ * & *\end{array}\right],
\]
where $RdR=R$. Using the fact that  $R$ is a Hermite ring,  we obtain that

\[
\left[\begin{array}{cc} d& *\\ * & *\end{array}\right]S=
\left[\begin{array}{cc}z& 0\\ * & *\end{array}\right]
\]
for an invertible  matrix $S$. Taking into account that   $dR\subset zR$ and $RdR=R$, we get that  $R=RdR\subseteq RzR$, i.e.~$RzR=R$.
\end{proof}

For a   canonical diagonal reduction of  matrices over a Hermite ring $R$ (see \cite[Theorem~5.1]{7-kaplansky})  it is enough to consider  such  reduction for matrices of the form $\left[\begin{array}{cc}a& 0\\ b & c\end{array}\right]$, where $a,b,c\in R$.
Therefore,  according to  Lemma~\ref{Lemma:3}, for   a   canonical diagonal reduction of matrices over
a Hermite ring  $R$ with $D$-property  it is enough to consider such  reduction for matrices
$\left[\begin{array}{cc}a& 0\\ b & c\end{array}\right]$,
where
\[
RaR+RbR+RcR=R.
\]
Using  Theorem~\ref{T:12} we have  the following assertion.

\begin{theorem}\label{T:13}
Let $R$ be a Hermite ring with $D$-property   of almost stable range $1$. The ring $R$  is an elementary divisor ring if and only if for every $a, b,c\in R$ with  $RaR=R$  the matrix $\left[\begin{array}{cc} a& 0\\ b & c\end{array}\right]\in R^{2\times 2}$  has  a canonical diagonal reduction over $R$.
\end{theorem}

\phantom{----}

\vskip-1.3em

This result shows that the set of  $a\in R$ for which $RaR=R$ is crucial  for a  canonical \linebreak diagonal reduction of matrices over $R$.

{\bf Example.} {\it 
Let $R$ be a domain. An element $a\in R\setminus\{0,U(R)\}$ is called {\it  atom}  if $a$  has no proper factors. An element which is either a unit or a product of atoms is
called  {\it finite}.  Let  $R$ be a B\'ezout domain with $D$-condition of almost stable range $1$. If  for every  $a\in R$, the  condition $RaR=R$ implies that $a$ is a finite element, then $R$ is an elementary divisor ring. Indeed, according to Theorem~\ref{T:13}, we can consider  a matrix  $A$ of the form $\left[\begin{array}{cc} a& 0\\ b &
 c \end{array}\right]\in R^{2\times 2}$, where $RaR=R$ and  $a$ is a finite element.
There exist  $P, Q\in \GL_2(R)$ (see \cite[Theorem 6]{18-zabkom}) such that $PAQ=\left[\begin{array}{cc} z& 0\\ 0 & d\end{array}\right]$ in which $RdR\subseteq zR\cap Rz$.
}


\section*{Acknowledgements}
The current  article finishes the joint collaboration,
which had started prior to Professor Bohdan  Zabavsky's demise in August 2020.


\newpage

\begin{thebibliography}{999}

\bibitem{Vas}
Akalan E.,Va\v{s} L. \textit{Classes of almost clean rings}. Algebr. Represent. Theory 2013, \textbf{16} (3), 843--857.
doi: 10.1007/s10468-012-9334-6

\bibitem{1-bass}
Bass H. \textit{$K$-theory and stable algebra}.
Publ. Math. Inst. Hautes \'{E}tudes Sci. 1964, \textbf{22}, 5--60.

\bibitem{Bovdi_Zabavsky_2}
Bovdi V., Zabavsky B. \textit{Elementary divisor rings with Dubrovin-Komarnytsky property}.
 2025, arXiv:2508.17100 [math.RA] doi:10.48550/arXiv.2508.17100

\bibitem{Bovdi_Zabavsky}
Bovdi V., Zabavsky B. \textit{Reduction of matrices over simple Ore domains}.
Linear Multilinear Algebra 2020, \textbf{70} (4), 642--649. doi:10.1080/03081087.2020.1743632

\bibitem{Bovdi_Shchedryk}
Bovdi V.A., Shchedryk V.P. \textit{Commutative Bezout domains of stable range 1.5}.
Linear Algebra Appl. 2019, \textbf{568}, 127--134. doi:10.1016/j.laa.2018.06.012

\bibitem{Bovdi_Shchedryk_adecvate}
Bovdi V.A., Shchedryk V.P. \textit{Adequacy of nonsingular matrices over commutative principal ideal domains}.
arXiv:2209.01408 [math.RA]. doi:10.48550/arXiv.2209.01408

\bibitem{15-Bowtell_Cohn}
Bowtell A.J., Cohn P.M. \textit{Bounded and invariant elements in $2$-firs}.
Math. Proc. Cambridge Philos. Soc. 1971, \textbf{69} (1), 1--12. doi:10.1017/S0305004100046375

\bibitem{Calugareanu}
C\u{a}lug\u{a}reanu G. \textit{On unit stable range matrices}.
Ann. Univ. Ferrara Sez. VII Sci. Mat. 2024, \textbf{70} (1), 127--140. doi:10.1007/s11565-023-00461-w

\bibitem{Dopico}
Dopico F.M., Noferini V., Zaballa I. \textit{Rosenbrock's theorem on system matrices over elementary divisor domains}.
Linear Algebra Appl. 2025, \textbf{710}, 10--49. doi:10.1016/j.laa.2025.01.028


\bibitem{dubrovin_plus}
Dubrovin N.I. \textit{The projective limit of rings with elementary divisors}.
Math. USSR-Sb. 1984, \textbf{47} (1), 85--90.

\bibitem{14-dubrovin}
Dubrovin N.I. \textit{On rings with elementary divisors}.
Izv. Vyssh. Uchebn. Zaved. Mat. 1986, \textbf{11}, 14--20. (in Russian)

\bibitem{Gatalevich}
Gatalevych A.I. \textit{On adequate and generalized adequate duo rings, and duo rings of elementary divisors}.
Mat. Stud. 1998, \textbf{9} (2), 115--119.

\bibitem{Shchedryk_Gatalevich}
Gatalevych A.I., Shchedryk V.P. \textit{On adequacy of full matrices}.
Mat. Stud. 2023, \textbf{59} (2), 115--122.
doi: 10.30970/ms.59.2.115-122

\bibitem{Helmer}
Helmer O. \textit{The elementary divisor theorem for certain rings without chain condition}.
Bull. Amer. Math. Soc. (N.S.) 1943, \textbf{49} (4), 225--236.

\bibitem{7-kaplansky}
Kaplansky I. \textit{Elementary divisors and modules}.
Trans. Amer. Math. Soc. 1949, \textbf{66} (2), 464--491.

\bibitem{Lam_2}
Khurana D., Lam T.Y., Nielsen P.P., \v{S}ter J. \textit{Special clean elements in rings}.
J. Algebra Appl. 2020, \textbf{19} (11), 2050208.
doi:10.1142/S0219498820502084

\bibitem{Lam_3}
Khurana D., Lam T.Y., Nielsen P.P., Zhou Y. \textit{Uniquely clean elements in rings}.
Comm. Algebra 2015, \textbf{43} (5), 1742--1751.
doi:10.1080/00927872.2013.879158

\bibitem{19-lamdugas}
Lam T.Y., Dugas A.S. \textit{Quasi-duo rings and stable range descent}.
J. Pure Appl. Algebra 2005, \textbf{195} (3), 243--259.
doi:10.1016/j.jpaa.2004.08.011

\bibitem{6-mcGovern}
McGovern W. \textit{Neat rings}.
J. Pure Appl. Algebra 2006, \textbf{205} (2), 243--265.
doi:10.1016/j.jpaa.2005.07.012

\bibitem{4-mcGovern}
McGovern W. \textit{B\'ezout rings with almost stable range $1$}.
J. Pure Appl. Algebra 2008, \textbf{212} (2), 340--348.
doi: 10.1016/j.jpaa.2007.05.026

\bibitem{Moore_Steger}
Moore M., Steger A. \textit{Some results on completability in commutative rings}.
Pacific J. Math. 1971, \textbf{37} (2), \linebreak 453--460.

\bibitem{Nicholson_77}
Nicholson W.K. \textit{Lifting idempotents and exchange rings}.
Trans. Amer. Math. Soc. 1977, \textbf{229}, 269--278.
doi: 10.1090/S0002-9947-1977-0439876-2

\bibitem{Shchedryk_2017}
Shchedryk V.P. \textit{Bezout rings of stable range $1.5$ and the decomposition of a complete linear group into the product of its subgroups}.
Ukrainian Math. J. 2017, \textbf{69} (1), 138--147. doi:10.1007/s11253-017-1352-4
(translation of Ukra\"{\i}n. Mat. Zh. 2017, \textbf{69} (1), 113--120. (in Ukrainian))

\bibitem{Mon_Shchedryk}
Shchedryk V. Arithmetic of matrices over rings.
Akademperiodyka, Kyiv, 2021. 

\bibitem{Shchedryk_2005}
Shchedryk V.P. \textit{Some properties of primitive matrices over B\'ezout $B$-domain}.
Algebra Discrete Math. 2005, \textbf{4} (2), 46--57.

\bibitem{2-schedryk}
Shchedryk V.P. \textit{Bezout rings of stable range $1.5$}.
Ukrainian Math. J. 2015, \textbf{67} (6), 960--974. doi:10.1007/s11253-015-1126-9
(translation of Ukra\"{\i}n. Mat. Zh. 2015, \textbf{67} (6), 849--860. (in Ukrainian))

\bibitem{18-zabkom}
Zabavsky B.V. \textit{On noncommutative rings with elementary divisors}.
Ukrainian Math. J. 1990, \textbf{42} (6), 748--750. doi:10.1007/BF01058928
(translation of Ukra\"{\i}n. Mat. Zh. 1990, \textbf{42} (6), 847--850. (in Russian))

\bibitem{12-zabavsky}
Zabavsky B. Diagonal reduction of matrices over rings.
In: Mathematical Studies Monograph Series, 16.
VNTL Publishers, Lviv, 2012.

\bibitem{11-zabavsky}
Zabavsky B. \textit{Conditions for stable range of an elementary divisor rings}.
Comm. Algebra 2017, \textbf{45} (9), 4062--4066.
doi:10.1080/00927872.2016.1259418

\bibitem{3-zabavsky}
Zabavsky B. \textit{Rings of dyadic range $1$}.
J. Algebra Appl. 2019, \textbf{18} (11), 1950206.
doi:10.1142/S0219498819502062

\bibitem{10-zabgat}
Zabavsky B., Gatalevych A. \textit{A commutative Bezout $PM^*$ domain is an elementary divisor ring}.
Algebra Discrete Math. 2015, \textbf{19} (2), 295--301.

\bibitem{9-zabavsky}
Zabavsky B.V. \textit{Diagonal reduction of matrices over finite stable range rings}.
Mat. Stud. 2014, \textbf{41} (1), 101--108.
doi:10.30970/ms.41.1.101-108

\end{thebibliography}
\end{document}